\documentclass[11pt,a4paper,reqno]{article}
\usepackage{amsmath}
\usepackage{amsthm}
\usepackage{enumerate}
\usepackage[inline]{enumitem}
\usepackage{amssymb}

\usepackage{verbatim}
\usepackage{url}
\usepackage[latin1]{inputenc}

\usepackage{caption}
\usepackage{graphicx,color}
\def\Z{{\mathbb Z}}

\def\R{{\mathbb R}}

\def\Fc{{\mathcal F}}

\DeclareMathOperator{\cross}{Cross}
\def\occ{{\mathcal O}}

\def\iidocc{\occ_\Phi}

\def\be{\begin{equation}}
\def\ee{\end{equation}}
\def\bea{\begin{equation*}}
\def\eea{\end{equation*}}
\def\bal{\begin{aligned}}
\def\eal{\end{aligned}}

\def\eps{\varepsilon}

\def\Pbf{{\bf P}}
\def\Ebf{{\bf E}}

\def\Pr{{\mathbb P}}
\DeclareMathOperator{\E}{{\mathbb E}}

\DeclareMathOperator{\ind}{{\bf 1}}

\newtheorem{thm}{Theorem}

\newtheorem{prop}[thm]{Proposition}

\newtheorem{quest}{Question}

\theoremstyle{remark}

\newtheorem{ex}{Example}

\theoremstyle{definition}

\begin{document}

\title{Gilbert's disc model with geostatistical marking\thanks{The authors thank IMPA in Rio de Janeiro and the Department of Mathematics at Chalmers University of Technology and the University of Gothenburg for their hospitality during visits related to this work. This work was supported by grant 150804/2012-1 from the Brazilian CNPq and grant 637-2013-7302 from the Swedish Research Council (DA), as well as the Knut and Alice Wallenberg foundation (JT).}}

\date{\today}

\author{Daniel Ahlberg and Johan Tykesson}

\maketitle

\begin{abstract}
We study a variant of Gilbert's disc model, in which discs are positioned at the points of a Poisson process in $\R^2$ with radii determined by an underlying stationary and ergodic random field $\varphi:\R^2\to[0,\infty)$, independent of the Poisson process. This setting, in which the random field is independent of the point process, is often referred to as \emph{geostatistical marking}.
We examine how typical properties of interest in stochastic geometry and percolation theory, such as coverage probabilities and the existence of long-range connections, differ between Gilbert's model with radii given by some random field and Gilbert's model with radii assigned independently, but with the same marginal distribution. Among our main observations we find that complete coverage of $\R^2$ does not necessarily happen simultaneously, and that the spatial dependence induced by the random field may both increase as well as decrease the critical threshold for percolation.\\

\end{abstract}

\section{Introduction}

Suppose that transceivers are positioned in the plane according to a Poisson point process. Whether a pair of transceivers are able to communicate with each other depends on their distance apart and their individual strength. More precisely, each transceiver is assigned a random value (the communication range) independently from a given distribution, and two transceivers are assumed to be able to communicate when their distance apart is at most the sum of their communication ranges. Of central interests are the geometric properties of the random subset of the plane consisting of all points within the communication range of some transceiver. Under the name of \emph{Gilbert's disc model} or \emph{Poisson Boolean percolation}, this model has been widely studied from the perspectives of stochastic geometry~\cite{chistokenmec13,schwei08} and percolation theory~\cite{bolrio06a,meeroy96}.

In this paper we study the following natural modification of this model. It is conceivable that the variation in communication range between different transceivers in many situations is better explained by external factors rather than on individual variation. As an example, the communication range could reflect the topography of the landscape where the transceiver is positioned. A way to incorporate this feature in the above model is to let the communication range of the transceivers be given by an underlying non-negative, stationary and ergodic random field. The random field imposes a dependence structure among the range of distinct Poisson points, which may alter the characteristics of the set of points in the plane within the range of communication of some transceiver. When the random field and the Poisson point process are independent, which will be the case in this paper, then this construction is known as \emph{geostatistical marking}, see e.g.~\cite{illpenstosto08,schribdig04}.

The main objective of this paper is to illustrate, in a few examples, the effect spatial dependence of a random field may have on coverage and percolation properties in Gilbert's disc model. That is, we shall compare the behaviour of Gilbert's model with geostatistical marking to Gilberts's model with radii sampled independently from the same marginal distribution as the random field. The following two examples will be enlightening in this regard.
The random field in the first example exhibits slowly decaying spatial correlations, while the second exhibits much faster decay of correlations.\footnote{In both examples the Poisson process used to generate the random field is independent from the Poisson processes used to determine the positions of the discs in the resulting Gilbert model.}

\begin{ex}\label{ex:cylinder}
Given a configuration of Poisson cylinders in the plane with unit radii, a stationary and ergodic random field is obtained by assigning non-negative i.i.d.\ random variables to the cylinders, and assigning to each point in the plane the smallest value among the cylinders it is contained in and zero if not contained in a cylinder.
\end{ex}

\begin{ex}\label{ex:box-model}
Given a Poisson point process in the plane, form its Voronoi tessellation and assign non-negative i.i.d.\ random variables to the Voronoi cells. The field whose value at each point is given by the value of the Voronoi cell in which the point is contained is a stationary and ergodic random field.
\end{ex}

The two examples may seem a bit artificial. However, they illustrate well the interesting phenomena that occur with geostatistical marking. First, observe that statistics such as the proportion of discs with radii in a given interval, depend only on the marginal distribution of the random field, due to the ergodic theorem, and is thus not affected by geostatistical marking. Nevertheless, we obtain in Example~\ref{ex:cylinder} a model where Gilbert's model with geostatistical marking does not cover the whole plane, whereas Gilbert's model with i.i.d.\ radii with the same marginal distribution is completely covered, almost surely. Moreover, Example~\ref{ex:box-model} illustrates the fact that Gilbert's model with geostatistical marking may be both more and less likely, depending on some parameter, to exhibit long-range connections as compared to its i.i.d.\ counterpart.


We complement the above observations with a few fundamental results relating to coverage and percolation properties. For instance, we show that complete coverage under geostatistical marking implies complete coverage for i.i.d.\ radii with the same marginal distribution. We also provide criteria for the existence of a subcritical and supercritical regime of percolation. Since our main objective has been to illustrate the effect of geostatistical marking to coverage and percolation properties, we have restricted our attention to the two dimensional setting.
The two dimensional setting is also the most natural in applications.

Gilbert~\cite{gilbert61} introduced his disc model in the 1960s, with constant radii, as a model for a random planar network with long-range connections. In~\cite{gilbert65} he went on to consider coverage problems in a related model. To the best of our knowledge, percolation theoretical questions in the geostatistical marking model have not been studied before. However, various authors have in recent years considered versions of Gilbert's model where discs are located according to point processes other than Poisson; see for instance the review of B{\l}aszczyszyn and Yogeshwaran~\cite{blayog15}. We mention in particular~\cite{blayog13}, by the same authors, for a discussion on the effect of clustering in point processes on the existence of long-range connections in the resulting Gilbert model. Related in this respect is also the dynamical model of Benjamini and Stauffer~\cite{bensta13}.
However, the original inspiration, at least partially, for the current study was a continuum version of a growth model with paralyzing obstacles, studied by van den Berg, Peres, Sidoravicius and Vares~\cite{berpersidvar08}.


In this paper, we merely lay out the basics for Gilbert's disc model with geostatistical marking. We hope that future work will address some of the questions that we have left open. For instance, we have failed to provide a sufficient condition for complete coverage. 
Moreover, although radii statistics are unaffected by geostatistical marking, it is less clear how edge count and edge length statistics of the associated random geometric graph are affected. As these are fundamental objects in the literature on random geometric graphs, see e.g.~\cite{penrose03}, this is another possible direction of future study.

\subsection{Description of the model}

Let $\eta$ be homogeneous Poisson point process on $\R^2\times[0,1]$ with intensity $\lambda\ge0$. Formally, $\eta$ is a random element on the space $(\Omega,\Fc)$ of counting measures on $\R^2\times[0,1]$, equipped with 
the smallest sigma algebra such that the sets $\{\nu\in\Omega:\nu(B)=k\}$ are measurable for integers $k\ge0$ and Borel sets $B\subseteq\R^2\times[0,1]$.
We shall denote the law of $\eta$ by $\Pr_\lambda$, and in order to ease the notation, we henceforth identify the random measure $\eta$ with its support.

We will throughout this paper let $\varphi:\R^2\to[0,\infty)$ denote a stationary and ergodic random field. That is, $\varphi$ will be a random element on the space $(\overline\Omega,\overline\Fc)$ of nonnegative functions on $\R^2$, equipped with the sigma algebra generated by the evaluation maps $\varphi\mapsto\varphi(x)$, for $x\in\R^2$.
That $\varphi$ is {\bf stationary} means that the law of $\varphi$, in the following denoted by $\overline\Pr$, is invariant with respect to translations by any vector $x\in\R^2$, i.e.,
$$
\big(\varphi(y_1+x),\varphi(y_2+x),\ldots,\varphi(y_k+x)\big)\stackrel{d}{=}\big(\varphi(y_1),\varphi(y_2),\ldots,\varphi(y_k)\big)
$$
for all $y_1,y_2,\ldots,y_k\in\R^2$ and $k\ge1$. In addition, the field is said to be {\bf ergodic} if every event $A\in\overline\Fc$ that is invariant with respect to translations\footnote{That is that $\sigma_x^{-1}(A)=A$ for all $x\in\R^2$, where $\sigma_x:\overline\Omega\to\overline\Omega$ is the shift $\varphi(\,\cdot\,)\mapsto\varphi(\,\cdot-x)$.\label{foot:shift}} occur with probability either zero or one.
Occasionally we will need to assume that the random field satisfies a stronger mixing condition, and that its law is invariant with respect to right angle rotations. In these cases these assumptions will be explicitly stated.

Finally, we equip $\Omega\times\overline\Omega$ with the product sigma algebra, and let $\Pbf_\lambda$ denote the product measure $\Pr_\lambda\times\overline\Pr$.
We remark that some care is usually needed when working with function spaces to ensure that the events of interest are measurable. In this paper, however, all events depend on the value of the field in a countable number of places, which effectively constitutes a random measure. Questions of measurability is in this setting a straightforward issue.

When working with geostatistical marking we shall often identify $\eta$ with its projection onto $\R^2$.
However, by considering a Poisson point process on $\R^2\times[0,1]$, we will be able to construct Gilbert's disc model with radii determined by a random field and independently assigned on the same probability space. More precisely, for a random field $\varphi:\R^2\to[0,\infty)$ we will let $\Phi$ denote its marginal distribution function
$$
\Phi(x):=\overline\Pr(\varphi(0)\le x).
$$
Then, given $(\eta,\varphi)\in\Omega\times\overline\Omega$ we define two random partitions $(\occ,\R^2\setminus\occ)$ and $(\iidocc,\R^2\setminus\iidocc)$ of the plane into an {\bf occupied} and {\bf vacant} set as follows:
$$
\occ(\eta,\varphi):=\bigcup_{(x,z)\in\eta}B\big(x,\varphi(x)\big)\quad\text{and}\quad\iidocc(\eta):=\bigcup_{(x,z)\in\eta}B\big(x,\Phi^{-1}(z)\big),
$$
where $B(x,r)$ denotes the closed Euclidean ball with radius $r$ centered at $x$ and $\Phi^{-1}$ denotes the generalized inverse $\Phi^{-1}(z):=\inf\{x\in\R:\Phi(x)\ge z\}$.\footnote{Recall that if $F^{-1}$ is the generalized inverse of some cumulative distribution function $F$ and $U$ is uniformly distributed on the interval $[0,1]$, then $F^{-1}(U)$ is distributed as $F$.} We have in $(\occ,\R^2\setminus\occ)$ a realization of Gilbert's disc model with radii obtained from a stationary and ergodic random field with marginal distribution $\Phi$, and in $(\iidocc,\R^2\setminus\iidocc)$ a realization of Gilbert's disc model with independent radii sampled from the same distribution $\Phi$.

In stochastic geometry one is typically interested in covering probabilities of compact sets and complete coverage of the whole plane. From a percolation perspective we have above all an interest in long-range connections and the existence of unbounded connected components, in which case we say that the model {\bf percolates}. Due to the ergodic nature of the random field, both the events of complete coverage and percolation in Gilbert's disc model with (or without) geostatistical marking are 0--1 events (see Proposition~\ref{prop:zero-one}). As it turns out, complete coverage is characterized by the marginal coverage probability in the sense that
$$
\Pbf_\lambda(0\in\occ)=1\quad\Leftrightarrow\quad\Pbf_\lambda(\occ=\R^2)=1.
$$
Moreover, complete coverage is a property of the underlying random field and irrelevant of the density $\lambda>0$ of the overlying Poisson process (see Proposition~\ref{prop:cc}).
The existence of an unbounded connected component in $\occ$ will, on the other hand, strongly depend on the value of $\lambda$. Percolation is characterized by a positive probability of the origin pertaining to an unbounded component. That is,
$$
\Pbf_\lambda\big(0\stackrel{\occ}{\leftrightarrow}\infty\big)>0\quad\Leftrightarrow\quad\Pbf_\lambda\big(\exists\text{ unbounded component}\big)=1,
$$
where $\{0\stackrel{\occ}{\leftrightarrow}\infty\}$ is shorthand for the event that there is a continuous curve $\gamma\subseteq\occ$ that connects the origin to points arbitrarily far away.
That the percolation probability is monotone in $\lambda$ follows from a standard coupling argument, and it is therefore customary to identify the critical value at which the transition from non-percolation to percolation occurs as
$$
\lambda_c:=\inf\big\{\lambda\ge0:\Pbf_\lambda\big(0\stackrel{\occ}{\leftrightarrow}\infty\big)>0\big\}.
$$

For Gilbert's disc model with i.i.d.\ radii, complete coverage and percolation are characterized analogously, and the circumstances for complete coverage and the existence of a non-trivial phase transition are well understood. Indeed, Hall~\cite{hall85} showed that complete coverage of $\R^2$ occurs if and only if the radii distribution $\Phi$ has infinite second moment, and Gou\'er\'e~\cite{gouere08} proved that when $\Phi$ has finite second moment the model exhibits a phase transition at some parameter $\lambda_\Phi\in(0,\infty)$. A precise description of the phase transition of Gilbert's disc model with i.i.d.\ radii has recently been derived by the first author with Tassion and Teixeira~\cite{ahltastei18a,ahltastei18b}.

\subsection{Description of the results}

Complete coverage in Gilbert's model with geostatistical marking implies complete coverage in the model with i.i.d.\ radii for the same marginal distribution. In fact, a straightforward calculation shows that
\begin{equation}\label{eq:point-coverage}
\Pbf_\lambda(0\in\occ)\,\le\,\Pr_\lambda(0\in\iidocc)
\end{equation}
for every stationary and ergodic random field $\varphi:\R^2\to[0,\infty)$ (see Proposition~\ref{prop:g-comparison}). However, for other compact sets, such as line segments, the inequality between the coverage probabilities may be reversed (see Proposition~\ref{prop:line-coverage}). Moreover, while~\eqref{eq:point-coverage} implies that the so-called `covered volume fraction' is always largest in the i.i.d.\ setting, one should not be fooled to believe that long-range connections are always less likely in the setting of geostatistical marking than with independently assigned radii. Indeed, the inequality may here go both ways (see Proposition~\ref{prop:box-large}), and while $\lambda_\Phi<\infty$ for every non-degenerate radii distribution, it may happen that $\lambda_c=\infty$. For instance, the latter is the case for the field in which each point takes the value given by half the distance to the closest neighbouring cell in an underlying Poisson Voronoi tessellation.

The above illustrates the more intrinsic role played by the spatial dependence induced by the geostatistical marking, which we highlight below in a theorem. This theorem summarizes Propositions~\ref{prop:plane-coverage},~\ref{prop:line-coverage} and~\ref{prop:box-large}, and is the result of a more detailed study of the random fields described in Examples~\ref{ex:cylinder} and~\ref{ex:box-model}, which will be conducted in Sections~\ref{sec:cylinder} and~\ref{sec:box-model}.

\begin{thm}\label{thm:comparison}
Consider Gilbert's disc model with geostatistical marking.
\begin{enumerate}[label=\itshape\alph*),topsep=3pt,partopsep=1pt,itemsep=3pt,parsep=1pt]
\item For every $\alpha>0$ there exists a stationary and ergodic field whose marginal distribution has infinite moment of order $1+\alpha$
whereas $\Pbf_\lambda(0\in\occ)<1$ for all $\lambda>0$.
\item There is a stationary and ergodic field such that for sufficiently large $n$ the probability of coverage of a line segment of length $n$ is strictly larger than in the i.i.d.\ setting.
\item There exists a one parameter family of stationary and ergodic random fields which for different values of the parameter yield either of the following two cases:
$$
0<\lambda_\Phi<\lambda_c<\infty\quad\text{and}\quad0<\lambda_c<\lambda_\Phi<\infty.
$$
\end{enumerate}
\end{thm}

We recall that for Gilbert's model with i.i.d.\ radii the dichotomy between complete coverage and the existence of a non-trivial phase transition is explained by a finite second moment condition on the radii distribution. Finite second moment of the radii distribution is in the i.i.d.\ setting equivalent to decay of spatial correlations. The above theorem shows that moment conditions are no longer good indicators for the existence of a non-trivial phase transition for Gilbert's model with geostatistical marking. We shall instead provide a sufficient criterion in terms of estimates on spatial correlations.

Let $B_\infty(x,r)$ denote the $\ell_\infty$-ball (i.e.\ square) centred at $x$ with radius $r$.\footnote{Not to be confused with $B(x,r)$ which denotes the Euclidean disc.} Fix $\eps_0\in(0,1/5]$ and let
$$
\overline\pi(n):=\sup_{g_1,g_2}\big|\overline\E[g_1g_2]-\overline\E[g_1]\overline\E[g_2]\big|,
$$
where the supremum is taken over all functions $g_1,g_2:\overline\Omega\to[0,1]$ such that $g_1\in\sigma(\{\varphi(x):x\in B_\infty(0,(1+\eps_0/4)n)\})$ and $g_2\in\sigma(\{\varphi(x):x\in B_\infty(((2+\eps_0)n,0),(1+\eps_0/4)n)\})$. The condition $\overline\pi(n)\to0$ as $n\to\infty$ implies that the field is mixing, and is hence a stronger condition than mere ergodicity. Some of our arguments will require this stronger condition.

We note that $\overline\pi$ measures correlations in the field between well separated regions, but does not take into account the actual magnitude of the field in these regions. Indeed, $\overline\pi$ is unaffected by any transformation that redefines the magnitude of the field in an injective manner. We will therefore need to complement $\overline\pi$ with a measure that captures the spatial correlation caused by discs covering large areas.
For any compact set $K\subset\R^2$, let $\occ_K(\eta,\varphi):=\occ(\eta\cap K,\varphi)$. Set $K=B_\infty(0,n)$ and $K'=B_\infty(0,(1+\eps_0/4)n)$ and define
$$
\pi_\lambda(n):=\Pbf_\lambda\big(\occ\cap K\neq\occ_{K'}\cap K\big).
$$

When the field is bounded, then $\pi_\lambda(n)=0$ for all large $n$, and a simple comparison with the constant radii model shows that $\lambda_c>0$. When the field is bounded away from zero, we similarly have $\lambda_c<\infty$, regardless of the correlations in the field. However, in order to give a general condition for a non-trivial phase transition we shall require an estimate on spatial correlations.
In order to guarantee that $\lambda_c>0$ we require only a minimal assumption on both of the above measures of correlations. To deduce that $\lambda_c<\infty$ we require no condition on $\pi_\lambda$,
but instead add a condition on the structure of the field.

For $s,t>0$ we let $\cross(s,t)$ denote the event that the restriction of $\occ$ to the rectangle $[0,s]\times[0,t]$ contains a continuous curve $\gamma$ connecting the left and right sides $\{0\}\times[0,t]$ and $\{s\}\times[0,t]$. In other words, $\cross(s,t)$ is the event that there is an occupied horizontal crossing of the rectangle $[0,s]\times[0,t]$. For $M>0$ we let $\cross_M(s,t)$ denote the analogous event for $\varphi$ replaced by its truncation $\varphi_M:=\min\{\varphi,M\}$.

\begin{thm}\label{thm:nontrivial}
Consider Gilbert's disc model with geostatistical marking, and assume, in addition, that the law of random field is invariant with respect to right angle rotations.
\begin{enumerate}[label=\itshape\alph*),topsep=3pt,partopsep=1pt,itemsep=3pt,parsep=1pt]
\item If, for some $\lambda>0$, $\pi_\lambda(n)$ and $\overline\pi(n)$ tend to zero as $n$ tends to infinity, then $\lambda_c>0$.
\item If, for some $\alpha>0$ and $M<\infty$, we have $\overline\pi(n)\le(\alpha\log n)^{-(1+\alpha)}$ for $n\ge1$ and
\begin{equation}\label{eq:thm2}
\lim_{n\to\infty}\lim_{\lambda\to\infty}\Pbf_\lambda\big(\cross_M(3n,n)\big)=1,
\end{equation}
then $\lambda_c<\infty$.
\end{enumerate}
\end{thm}

The condition in~\eqref{eq:thm2} may seem puzzling at first, but is easily verifiable in a number of concrete examples. For instance, in Example~\ref{ex:cylinder} the condition holds whenever the distribution used to assign values to cylinders either has no atom at zero, or gives positive weight to values strictly larger than one. In Example~\ref{ex:box-model} the condition holds, in particular, whenever the distribution assigning weights to the cells gives less than half its mass to the value zero. As we explain at the end of Section~\ref{sec:nontriviality}, the condition in~\eqref{eq:thm2} is indeed also necessary for a large class of bounded random fields.

The point with the condition in~\eqref{eq:thm2} is, of course, to rule out that large regions at which the field attains the value zero could prevent percolation. In certain cases level sets of the random field is a natural notion to consider. It may then be useful to address the condition in~\eqref{eq:thm2} in terms of these level sets. For $\alpha>0$ and $n\ge1$, let $L_{\alpha,n}(\varphi)$ denote the event that there is a continuous curve $\gamma$ contained in the restriction $\Gamma$ of $\textup{int}\{x\in\R^2:\varphi(x)>\alpha\}$ to the rectangle $[0,3n]\times[0,n]$, connecting the left and right sides $\{0\}\times[0,n]$ and $\{3n\}\times[0,n]$. It is not hard to show\footnote{By compactness one may find a finite set of discs with radius at most $\alpha/2$ whose union contains $\gamma$ and is contained in $\Gamma$. If each disc contains a point of $\eta$, which is likely for large $\lambda$, then $\cross_\alpha(3n,n)$ occurs.} that if $\overline\Pr(L_{\alpha,n}(\varphi))\to1$ as $n\to\infty$, then~\eqref{eq:thm2} holds. 

We have had no intention in providing a complete description of Gilbert's disc model with geostatistical marking in this study. Instead we have aimed to describe how some of its features may differ from what is observed in the i.i.d.\ setting, and for that reason focused on the planar setting. Our results take a mere first step towards a more general understanding of the model, which we hope will be better understood in future work. Recent work on Gilbert's model with i.i.d.\ radii give a detailed description of its phase transition~\cite{ahltastei18a,ahltastei18b}. These works certainly hint at what to expect also in the case of geostatistical marking, but we expect that the spatial dependence in the random field will pose certain challenges that one would need to overcome. In higher dimensions one should be able to extend several of the results obtained here in two dimensions. For instance, it seems likely that techniques used by Gou\'er\'e~\cite{gouere08} could be used to obtains a version of Theorem~\ref{thm:nontrivial} above in dimensions higher than two.

\section{Fundamentals}

We will in this section go through some preliminary observations that will be important for the remainder of the paper. We aim to work under minimal assumptions on the random field used for the geostatistical marking. For instance, although the specific examples we consider typically are positively correlated, for none of the arguments or techniques we use will this be a requirement. We remark, however, that most natural notions of positive association in the random field implies positive association of the resulting measure $\Pbf_\lambda$. For instance, one such natural notion is to call a random field $\varphi:\R^2\to[0,\infty)$ {\bf positively associated} if for all monotone functions $g_1,g_2:\overline\Omega\to[0,1]$, that is $g(\varphi)\ge g(\varphi')$ whenever $\varphi\ge\varphi'$ point-wise, we have
$$
\overline\E[g_1(\varphi)g_2(\varphi)]\ge\overline\E[g_1(\varphi)]\overline\E[g_2(\varphi)].
$$





\subsection{A zero-one law}

Complete coverage and the existence of an unbounded occupied component are examples of events that are invariant with respect to translations in the plane. Events which are measurable with respect to a Poisson process (that is, events in $\Fc$) and invariant under translations are well-known to occur with probability either $0$ or $1$, see e.g.~\cite[Proposition~2.8]{meeroy96}. We shall see that also in our setting, in which the point of the Poisson process are marked by an independent random field,
invariant events are 0-1 events.

Let $T({\mathbb R}^2)$ be the group of translations on ${\mathbb R}^2$. Ergodicity of the random field has the following equivalent characterization: For all bounded measurable functions $g_1,g_2:\overline\Omega\to\R$ we have
\begin{equation}\label{def:ergodicity}
\lim_{t\to\infty}\frac{1}{t^2}\int_{[0,t]^2}\overline\E[g_1\circ\sigma_x\cdot g_2]\,dx\,=\,\overline\E[g_1]\overline\E[g_2],
\end{equation}
where $\sigma_x$ denotes the shift along the vector $x\in\R^2$ (see Footnote~\ref{foot:shift}). That~\eqref{def:ergodicity} is equivalent to ergodicity is well-known, and follows, for instance, by straightforward modifications of the proof of Proposition~2.5 in~\cite{meeroy96}.

With a slight abuse of notation, we shall also let $\sigma_x:\Omega\times\overline\Omega\to\Omega\times\overline\Omega$ denote the shift
$$
\big(\eta(\,\cdot\,),\varphi(\,\cdot\,)\big)\mapsto\big(\eta(\,\cdot-x),\varphi(\,\cdot-x)\big),
$$
where $x\in\R^2$.
We say that an event $A\subseteq\Omega\times\overline\Omega$ is {\bf invariant under diagonal action} of $T({\mathbb R}^2)$ if for all $x\in\R^2$ we have $\sigma_x^{-1}(A)=A$.

\begin{prop}\label{prop:zero-one}
If $A$ is invariant under diagonal action of $T({\mathbb R}^2)$, then ${\bf P}_{\lambda}(A)\in \{0,1\}.$
\end{prop}

\begin{proof}
The proof is similar to the proof of Lemma~2.6 in~\cite{hagjon06}. Given $x\in\R^2$ and $n\ge1$ let $\mathcal{F}_{x,n}$ denote the sigma algebra generated by the restriction of $\eta$ and $\varphi$ to $B_\infty(x,n)$. Set
$$
I_{x,n}:=\ind_{\{{\mathbf P}_{\lambda}(A|\mathcal{F}_{x,n})>1/2\}}.
$$
Then, by Levy's 0-1 law, we have
\begin{equation}\label{e.ergod4}
\lim_{n\to \infty}I_{x,n}=\ind_A\quad\text{${\mathbf P}_{\lambda}$-almost surely}.
\end{equation}
Since $A$ is invariant under diagonal action of $T({\mathbb R}^2)$, the laws of $(I_{0,n},\ind_A)$ and $(I_{x,n},\ind_A)$ are the same.
Consequently, uniformly in $x\in\R^2$, we have
\begin{equation}\label{e.ergod1}
\lim_{n\to \infty} {\mathbf P}_{\lambda}(I_{0,n}=I_{x,n}=\ind_A)\,\ge\,\lim_{n\to\infty}\big[1-2\Pbf_\lambda(I_{0,n}\neq\ind_A)\big]
\,=\,1.
\end{equation}

For $x$ outside of $B_\infty(0,2n)$, we get for any $i,j\in\{0,1\}$ that
$$
\Pbf_\lambda(I_{0,n}=i,I_{x,n}=j|\varphi)=\Pbf_\lambda(I_{0,n}=i|\varphi)\Pbf_\lambda(I_{x,n}=j|\varphi).
$$
Since $\overline\Pr$ is assumed to be ergodic, it follows from~\eqref{def:ergodicity} that the limit (for fixed $n$)
$$
\lim_{t\to\infty}\frac{1}{t^2}\int_{[0,t]^2}\Pbf_\lambda(I_{0,n}=i,I_{x,n}=j)\,dx=\lim_{t\to\infty}\frac{1}{t^2}\int_{[0,t]^2}\overline\E\big[\Pbf_\lambda(I_{0,n}=i|\varphi)\Pbf_\lambda(I_{x,n}=j|\varphi)\big]\,dx
$$
equals ${\mathbf P}_{\lambda}(I_{0,n}=i){\mathbf P}_{\lambda}(I_{0,n}=j)$. On the other hand, due to~\eqref{e.ergod1}, we have for each $\delta>0$ and all large $n$ that
$$
\delta\,>\,\lim_{t\to\infty}\frac{1}{t^2}\int_{[0,t]^2}\Pbf_\lambda(I_{0,n}=1,I_{x,n}=0)\,dx\,=\,{\mathbf P}_{\lambda}(I_{0,n}=1){\mathbf P}_{\lambda}(I_{0,n}=0).
$$
That is, sending $n$ to infinity and $\delta$ to zero leaves us with
$$
0=\Pbf_\lambda(A)\big[1-\Pbf_\lambda(A)\big],
$$
and hence that ${\mathbf P}_{\lambda}(A)\in\{0,1\}$, as required.
\end{proof}


\subsection{Complete coverage}

The following gives a condition for complete coverage based on the coverage probability of the origin. The analogous statement is of course well known in the i.i.d. setting, see~\cite{hall85}.

\begin{prop}\label{prop:cc}
The following statements are equivalent:
\begin{enumerate}[label=\itshape\alph*),topsep=3pt,partopsep=1pt,itemsep=3pt,parsep=1pt]
\item $\Pbf_\lambda(0\in\occ)=1$ for some $\lambda>0$.
\item $\Pbf_\lambda(0\in\occ)=1$ for every $\lambda>0$.
\item $\Pbf_\lambda(\occ=\R^2)=1$ for every $\lambda>0$.
\end{enumerate}
\end{prop}

\begin{proof}
We shall prove that~\emph{a)} implies~\emph{b)} implies~\emph{c)}. That~\emph{c)} implies~\emph{a)} is trivial. Assume hence that $\Pbf_\lambda(0\in\occ)=1$ for some $\lambda>0$. Using that the superpositioning of two independent Poisson processes is a Poisson process and Jensen's inequality, we obtain
$$
\Pbf_{2\lambda}(0\not\in\occ)\,=\,\overline\E\big[\Pr_\lambda(0\not\in\occ|\varphi)^2\big]\,\ge\,\Pbf_\lambda(0\not\in\occ)^2.
$$
A standard coupling argument shows that the probability $\Pbf_\lambda(0\in\occ)$ is monotone in $\lambda$. Hence, if $\Pbf_\lambda(0\not\in\occ)>0$ for some $\lambda>0$, then it will be for arbitrarily large values of $\lambda$ too. So we must have $\Pbf_\lambda(0\in\occ)=1$ for all $\lambda>0$.

Now assume that $\Pbf_\lambda(0\in\occ)=1$ for all $\lambda>0$. Then, for any compact set $K\subseteq\R^2$,
$$
1\,=\,\Pbf_\lambda(0\in\occ)\,\le\,\Pbf_\lambda\big(0\in\occ(\eta\cap K,\varphi)\big)+\Pbf_\lambda\big(0\in\occ(\eta\cap K^c,\varphi)\big).
$$
The first of the two probabilities can be made arbitrarily small by decreasing $\lambda$. Due to monotonicity in $\lambda$ we find that $\Pbf_\lambda\big(0\in\occ(\eta\cap K^c,\varphi)\big)\ge1-\eps$ for every $\eps>0$. That is, $\Pbf_\lambda\big(0\in\occ(\eta\cap K^c,\varphi)\big)=1$ for every compact set $K$. This shows that the origin is covered by arbitrarily large discs, and hence infinitely many, almost surely. Infinitely many of these points will belong to one of the four quadrants. Let $A$ denote the event that infinitely many discs with center in the first quadrant overlap the origin. We may assume, without loss of generality, that $\Pbf_\lambda(A)>0$. Let $A_x$ denote the translate of $A$ along the vector $x\in\R^2$. By the ergodic theorem, $A_x$ will occur for infinitely many $x$ in the third quadrant with probability one. Note that if $A_x$ occurs for some $x$ in the third quadrant, then $[0,1]^2$ is covered almost surely.
The conclusion now follows by tiling the plane by unit squares.
\end{proof}

\subsection{Intersection probabilities}

We show below that the probability of covering the origin is never larger in the presence of a random field than in the i.i.d.\ setting. In fact, we shall prove a more general statement, which compares the probabilities that a compact set is intersected by the disc of a Poisson point centered far away.
Probabilities of this kind are typically easy to bound in the i.i.d.\ setting, and the observation made here is that the i.i.d.\ setting dominates that of a stationary ergodic field in the following sense.

\begin{prop}\label{prop:g-comparison}
Let $K$ and $K'$ be compact subsets of $\R^2$. Then,
$$
\Pbf_\lambda\Big(\exists x\in\eta\setminus K':B(x,\varphi(x))\cap K\neq\emptyset\Big)\,\le\,\Pr_\lambda\Big(\exists (x,z)\in\eta\setminus K':B(x,\Phi^{-1}(z))\cap K\neq\emptyset\Big).
$$
\end{prop}

\begin{proof}
Let
\begin{equation*}
\begin{aligned}
\bar\eta&:=\big\{x\in\eta\setminus K':B(x,\varphi(x))\cap K\neq\emptyset\big\},\\
\bar\eta_\Phi&:=\big\{(x,z)\in\eta\setminus K':B(x,\Phi^{-1}(z))\cap K\neq\emptyset\big\}.
\end{aligned}
\end{equation*}
The projection of $\bar\eta_\Phi$ onto $\R^2$ is a thinned Poisson point process with density function $x\mapsto\lambda\,\Pr\big(B(x,\Phi^{-1}(U))\cap K\neq\emptyset\big)\ind_{\{x\not\in K'\}}$, where $U$ is uniform on $[0,1]$. Hence, $|\bar\eta_\Phi|$ is Poisson distributed with parameter (given that it is finite)
$$
\lambda\int_{\R^2\setminus K'}\Pr\big(B(x,\Phi^{-1}(U))\cap K\neq\emptyset\big)\,dx.
$$
In particular, $\Pr_\lambda(\bar\eta_\Phi=\emptyset)=\exp\big(-\lambda\int_{\R^2\setminus K'}\Pr(B(x,\Phi^{-1}(U))\cap K\neq\emptyset)\,dx\big)$.

Conditioned on $\varphi$ also $\bar\eta$ is a thinned Poisson point process, this time with density
$$
x\mapsto\lambda\ind_{\{x\in\R^2\setminus K':B(x,\varphi(x))\cap K\neq\emptyset\}}.
$$
Hence, the conditional law of $|\bar\eta|$, given the field $\varphi$, is Poisson distributed with parameter $\lambda\int_{\R^2\setminus K'}\ind_{\{B(x,\varphi(x))\cap K\neq\emptyset\}}\,dx$, and hence
$$
\Pr_\lambda(\bar\eta=\emptyset|\varphi)\,=\,\exp\bigg(-\lambda\int_{\R^2\setminus K'}\ind_{\{B(x,\varphi(x))\cap K\neq\emptyset\}}\,dx\bigg).
$$
Using Jensen's inequality and Fubini's theorem we may thus deduce that
$$
\Pbf_\lambda(\bar\eta=\emptyset)\,=\,\overline\E\big[\Pr_\lambda(\bar\eta=\emptyset|\varphi)\big]\,\ge\,\exp\bigg(-\lambda\int_{\R^2\setminus K'}\overline\Pr\big(B(x,\varphi(x))\cap K\neq\emptyset\big)\,dx\bigg).
$$
However, the right-hand side coincides with $\Pr_\lambda(\bar\eta_\Phi=\emptyset)$, as required.
\end{proof}

Proposition~\ref{prop:g-comparison} has several interesting consequences:
\begin{enumerate}[label=\itshape(\roman*)]
\item For $K=\{0\}$ and $K'=\emptyset$ the statement of the proposition is reduced to
$$
\Pbf_\lambda(0\in\occ)\le\Pr_\lambda(0\in\iidocc),
$$
which is the statement of~\eqref{eq:point-coverage}. From the proof we obtain the well-known expression
$$
\Pr_\lambda(0\in\iidocc)=1-\exp(-2\pi\lambda\E[\Phi^{-1}(U)^2]).
$$
\item The expected Lebesgue measure of points in a unit square covered at intensity $\lambda$ is given by $\Pbf_\lambda(0\in\occ)$ as a consequence of Fubini's theorem. The \emph{covered volume fraction}, defined as the almost sure limit
$$
\lim_{n\to\infty}\frac{1}{4n^2}\int_{[-n,n]^2}\ind_{\{x\in\occ\}}\,dx,
$$
is by the ergodic theorem equal to $\Pbf_\lambda(0\in\occ)$, and is thus maximized for i.i.d.\ radii.
\item Let $K=B_\infty(0,n)$ and $K'=B_\infty(0,(1+\eps/4)n)$. When $\Phi$ has finite second moment it is in the i.i.d.\ setting straightforward to check that
$$
\Pr_\lambda\big(\exists (x,z)\in\eta\setminus K':B(x,\Phi^{-1}(z))\cap K\neq\emptyset\big)\to0\quad\text{as }n\to\infty.
$$
Consequently, finite second moment of the marginal distribution of $\varphi$ is sufficient for the decay of the spatial correlations measured by $\pi_\lambda(n)$. That is, for any $\lambda>0$,
$$
\overline\E[\varphi(0)^2]<\infty\quad\Rightarrow\quad\lim_{n\to\infty}\pi_\lambda(n)=0.
$$
\end{enumerate}

\subsection{Measures of spatial correlations}

We described in the introduction two measures on spatial correlations, one measuring the spatial correlation in the random field and another which takes into account the correlations induced by large values of the field. This distinction is often useful, but we shall here briefly discuss an alternative notion of correlations that combines the two into one. This notion may be more natural in other instances.

The measure $\Pbf_\lambda$ induces a measure on subsets of $\R^2$, and the associated measurable space $(\widetilde\Omega,\widetilde\Fc)$, via the canonical projections $Y_x=\ind_{\{x\in\occ\}}$, for $x\in\R^2$. We define an associated measure on spatial correlations as follows: Fix $\eps_0\in(0,1/5]$ as above. For every $n\ge1$, let
$$
\rho_\lambda(n):=\sup_{f_1,f_2}\big|\Ebf_\lambda[f_1f_2]-\Ebf_\lambda[f_1]\Ebf_\lambda[f_2]\big|,
$$
where the supremum is taken over all functions $f_1,f_2:\widetilde\Omega\to[0,1]$ such that $f_1\in\sigma\big(Y_x:x\in B_\infty(0,n)\big)$ and $f_2\in\sigma\big(Y_x:x\in((2+\eps_0)n,0)+B_\infty(0,n)\big)$.

The spatial correlations measured by $\rho_\lambda$ come up naturally in many instances. However, there are occasionally advantages in formulating an hypothesis in terms of $\pi_\lambda$ and $\overline\pi$ instead of $\rho_\lambda$. For instance, it is unclear how $\rho_\lambda$ depends on $\lambda$, whereas $\pi_\lambda$ is clearly increasing in $\lambda$.
Given a bound on $\pi_\lambda$ and $\overline\pi$, it transforms into a bound on $\rho_\lambda$ as follows.

\begin{prop}\label{prop:cor-bound}
For every $\lambda>0$ and $n\ge1$ we have $\rho_\lambda(n)\le 4\pi_\lambda(n)+\overline\pi(n)$.
\end{prop}

\begin{proof}
Fix $n\ge1$, and let $K_1=B_\infty(0,n)$ and $K_2=((2+\eps_0)n,0)+K_1$. Let $f_1,f_2:\widetilde\Omega\to[0,1]$ be two functions such that $f_1\in\sigma(Y_x:x\in K_1)$ and $f_2\in\sigma(Y_x:x\in K_2)$. Also, let $K_1'=B_\infty(0,(1+\eps_0/4)n)$ and $K_2'=((2+\eps_0)n,0)+K_1'$, and let $G_i=\{\occ\cap K_i=\occ_{K_i'}\cap K_i\}$. On the event $G_i$ we have $f_i(\occ)=f_i(\occ_{K_i'})$, so
$$
\big|\Ebf_\lambda[f_1(\occ)f_2(\occ)]-\Ebf_\lambda[f_1(\occ_{K_1'})f_2(\occ_{K_2'})]\big|\le2\pi_\lambda(n),
$$
and similarly $\big|\Ebf_\lambda[f_1(\occ)]\Ebf_\lambda[f_2(\occ)]-\Ebf_\lambda[f_1(\occ_{K_1'})]\Ebf_\lambda[f_2(\occ_{K_2'})]\big|\le2\pi_\lambda(n)$.

Since $K_1'$ and $K_2'$ are disjoint, $f_1(\occ_{K_1'})$ and $f_2(\occ_{K_2'})$ are conditionally independent given $\varphi$, so that
$$
\E_\lambda[f_1(\occ_{K_1'})f_2(\occ_{K_2'})|\varphi]=\E_\lambda[f_1(\occ_{K_1'})|\varphi]\E_\lambda[f_2(\occ_{K_2'})|\varphi],
$$
almost surely. The two factors of the right-hand side are $[0,1]$-valued variables on $\overline\Omega$, so by definition of $\overline\pi(n)$ we obtain
$$
\Big|\overline\E\big[\E_\lambda[f_1(\occ_{K_1'})|\varphi]\E_\lambda[f_2(\occ_{K_2'})|\varphi]\big]-\Ebf_\lambda[f_1(\occ_{K_1'})]\Ebf_\lambda[f_2(\occ_{K_2'})]\Big|\le\overline\pi(n).
$$
Summing up the error estimates, via the triangle inequality, leaves us with
$$
\rho_\lambda(n)=\sup_{f_1,f_2}\big|\Ebf_\lambda[f_1(\occ)f_2(\occ)]-\Ebf_\lambda[f_1(\occ)]\Ebf_\lambda[f_2(\occ)]\big|\le4\pi_\lambda(n)+\overline\pi(n),
$$
as required.
\end{proof}

Note that when the plane is almost surely completely covered, then $\rho_\lambda(n)=0$ but $\pi_\lambda(n)=1$ for all $\lambda>0$ and $n\ge1$. Nevertheless, we would still expect that decay of correlations in the sense measured by $\rho_\lambda$, should typically imply some sort of control on the correlations captures by $\pi_\lambda$ and $\overline\pi$.

\begin{quest}
In what sense does a converse to Proposition~\ref{prop:cor-bound} hold?
\end{quest}


\section{Coverage probabilities}\label{sec:cylinder}

From Proposition~\ref{prop:g-comparison} we learn that i.i.d.\ assignment of radii maximize intersection probabilities. When it comes to coverage probabilities the situation is different. We shall in this section study a random field constructed from a Poisson cylinder process in the plane, and see that coverage probabilities may for this field both exceed and be inferior to the i.i.d.\ setting with the same marginal distribution.

A Poisson line process $Y$ of intensity $u$ in ${\mathbb R}^2$ may be constructed in the following standard fashion: Let $\ell_{\theta}$ be the half-line emanating from the origin at angle $\theta$ with the first coordinate axis. For $(\theta,x)\in [0,2\pi)\times [0,\infty)$ let $\ell(\theta,x)$ be the line perpendicular to $\ell_{\theta}$ that intersects $\ell_\theta$ at distance $x$ from the origin. Next, consider a homogeneous Poisson point process $X$ on $[0,2\pi)\times [0,\infty)$ with intensity $u>0$. The Poisson line process $Y$ is obtained by interpreting $X$ as the random collection of lines $\{\ell(\theta,x):(\theta,x)\in X\}$. Similarly we obtain a marked Poisson line process $Y=\{(\ell(\theta,x),z):(\theta,x,z)\in X\}$ from a Poisson point process $X$ on $[0,2\pi)\times[0,\infty)\times[0,1]$ with intensity $u$.

For $\ell\in Y$, let $C(\ell)$ denote the bi-infinite solid closed cylinder with base-radius $r>0$ centered around the infinite line $\ell$, and set
$$
{\mathcal L}:=\bigcup_{(\ell,z)\in Y}C(\ell).
$$
Let $F$ be the distribution function of some probability measure on $[0,\infty)$. We define a random field $\varphi:\R^2\to[0,\infty)$ by letting $\varphi(x)=0$ if $x\not\in {\mathcal L}$, and $\varphi(x)=\inf\{F^{-1}(z):(\ell,z)\in Y,x\in C(\ell)\}$ if $x\in {\mathcal L}$. For each $F$ we obtain a a family of random fields with two parameters $u$ and $r$.

Geometric and percolative properties of the random set ${\mathcal L}$ and its complement have been studied in dimensions three and higher in~\cite{tykwin12},~\cite{hilsidtei15} and~\cite{brotyk16}. There, the process is referred to as the \emph{Poisson cylinder model}. A characteristic feature of this model is that it exhibits long-range spatial dependence, see~\cite[Lemma~3.1]{tykwin12} for details.

It is easy to see for every $u>0$ and $r>0$ that the marginal distribution $\Phi$ of the random field $\varphi$ has finite moment of a given order if and only if $F$ does. More precisely, if $U$ is uniform on $[0,1]$ and $\alpha>0$, then
$$
\overline\Pr\big(\text{exactly one cylinder contains }0\big)\E[F^{-1}(U)^\alpha]\,\le\,\overline\E[\varphi(0)^\alpha]\,\le\,\E[F^{-1}(U)^\alpha],
$$
assuming the moments exist. The following proposition shows that for every $\alpha>0$ there exists a random field with $\overline\E[\varphi(0)^{1+\alpha}]=\infty$, but for which $\Pbf_\lambda(\occ=\R^2)=0$ for all $\lambda>0$.

\begin{prop}\label{prop:plane-coverage}
Assume that $F$ has finite mean. Then, for any $u>0$ and $r>0$ we have
$$
\Pbf_\lambda(0\in \occ)<1.
$$
\end{prop}

\begin{proof}
We make the initial observation that
$$
\occ \subseteq \widetilde{\occ}:= \bigcup_{(\ell,z)\in Y}C(\ell,F^{-1}(z)+r).
$$
We also observe that the number of lines $(\ell,z)\in Y$ such that $0\in C(\ell,F^{-1}(z)+r)$ is a Poisson random variable with mean
$$
2\pi u\int_{0}^{\infty} \Pr\big(F^{-1}(U)\ge t-r\big)\,dt\,=\,2\pi\Big(ur+u\int_{0}^{\infty}\big(1-F(t)\big)\,dt\Big)\,=\,2\pi u\big(r+\E[F^{-1}(U)]\big),
$$
where $U$ is again uniform on $[0,1]$. Hence,
$$
\Pbf_{\lambda}(0\in \occ)\,\le\, \Pbf_{\lambda}(0\in \widetilde{\occ})\,=\,1-\exp\big(-2\pi u(r+\E[F^{-1}(U)])\big)\,<\,1,
$$
whenever $F$ has finite mean. Note that this bound does not depend on $\lambda$.
\end{proof}

When $F$ has infinite mean, then the origin is covered with probability one in the above model. It seems plausible that this is a consequence of a more general fact.

\begin{quest}
Does $\overline\E[\varphi(0)]=\infty$ imply that $\Pbf_\lambda(\occ=\R^2)=1$?
\end{quest}

Next, we obtain an example of a field for which the probability of containing a line segment is larger with geostatistical marking than for i.i.d.\ assignment of radii with the same marginal distribution. Let $L_s$ be the line segment between $(0,0)$ and $(s,0)$.

\begin{prop}\label{prop:line-coverage}
Let $r=2$ and let $F$ correspond to a point mass at 1. Then there exists $u_0>0$ such that for every $u\in (0,u_0)$ there is $s_0=s_0(u)>0$ such that
$$
\Pbf_{\lambda}(L_s\subseteq \occ)> {\mathbb P}_{\lambda}(L_s\subseteq \occ_{\Phi})\quad\text{for all }s\ge s_0.
$$
\end{prop}

\begin{proof}
Let $A_s$ be the event that there is a line in $Y$ intersecting $B(0,1)$ and $B((0,s),1)$. A straightforward calculation, see~\cite[Lemma~3.1]{tykwin12}, shows that for some $c>0$ we have
\begin{equation}\label{e.arbound}
\overline\Pr_u(A_s)\ge \frac cs\,\overline\Pr_u(0\in {\mathcal L})
\end{equation}
uniformly in $u$. Observe that on $A_s$, the 1-neighborhood of $L_s$ is contained in ${\mathcal L}$. Hence, conditioned on $A_s$, the probability that $L_s$ is contained in $\occ(\eta,\varphi)$ equals the probability that $L_s$ is contained in the occupied set in Gilbert's model with unit radii. That is,
\begin{equation}\label{e.areq4}
\Pbf_{\lambda}(L_s\subseteq \occ | A_s)={\mathbb P}_{\lambda}(L_s\subseteq \occ_F).
\end{equation}

Another calculation, see~\cite[Lemma~3.4]{benjonschtyk09}, shows that there is a continuous function $\alpha:(0,\infty)\to(0,\infty)$ with the properties that $\alpha(\lambda)\to 0$ as $\lambda\to \infty$, $\alpha(\lambda)\to \infty$ as $\lambda\to 0$, and such that
\begin{equation}\label{e.areq3}
c' \exp(-\alpha(\lambda) s)\, \le\, {\mathbb P}_{\lambda}(L_s\subseteq \occ_F)\, \le\, \frac{1}{c'}\exp(-\alpha(\lambda) s)
\end{equation}
for some $c'>0$. Combining~\eqref{e.arbound},~\eqref{e.areq4} and~\eqref{e.areq3} we obtain that
\begin{equation}\label{e.b1}
\Pbf_{\lambda}(L_s\subseteq \occ)\,\ge\,  \Pbf_{\lambda}(L_s\subseteq \occ|A_s)\overline\Pr_u(A_s)\,\ge\, \frac{cc'}{s} \exp(-\alpha(\lambda) s) \overline\Pr_u(0\in {\mathcal L}).
\end{equation}

We now observe that the law of $\occ_{\Phi}(\eta)$ equals that of Gilbert's model with unit radii at density $\lambda'(u):=\lambda \overline\Pr_u(0\in {\mathcal L})$. That is,
\begin{equation}\label{e.b2}
{\mathbb P}_{\lambda}(L_s\subseteq\occ_{\Phi})\,=\,{\mathbb P}_{\lambda'}(L_s\subseteq \occ_F)\,\le\,\frac{1}{c'} \exp(-\alpha(\lambda') s),
\end{equation}
where the upper bound again is due to~\eqref{e.areq3}. Since $\lambda'(u)\to 0$ as $u\to 0$, we can choose $u_0$ small so that $\alpha(\lambda'(u))>\alpha(\lambda)$ for all $u\le u_0$. Combining~\eqref{e.b1} and~\eqref{e.b2}, we see that for some $s_0=s_0(u)$ we have
$$
{\mathbb P}_{\lambda}(L_s\subseteq\occ_{\Phi})<\Pbf_{\lambda}(L_s\subseteq \occ)
$$
for all $s\ge s_0$, as required.
\end{proof}

\section{Non-triviality of the critical threshold}\label{sec:nontriviality}

The aim of this section is to prove Theorem~\ref{thm:nontrivial}. The first step will be to derive a so-called `finite-size' criterion, which is a technique developed in the early 1980s, see~\cite{russo81,kesten82,aizchachafrorus83}, but also~\cite{ahltastei18a} for a recent example. We assume throughout this section that the law of the random field $\varphi$ is invariant with respect to rotations by right angles. 
Since the random field is invariant with respect to translations and rotations by right angles, so is $\cross(s,t)$. This will be used repeatedly.

\begin{prop}\label{prop:finite-size}
Suppose that the law of the random field is invariant with respect to rotations by right angles. Let $I\subseteq[0,\infty)$ be some interval and assume that $\rho_\lambda(n)\to0$ as $n\to\infty$ uniformly over $\lambda\in I$. Then, there exists $N=N(I)$ such that for every $\lambda\in I$ the following are true:
\begin{enumerate}[label=\itshape\alph*)]
\item If there exists $n\ge N$ such that $\Pbf_\lambda(\cross(3n,n))>1-1/200$, then
$$
\lim_{n\to\infty}\Pbf_\lambda(\cross(3n,n))=1.
$$
\item If there exists $n\ge N$ such that $\Pbf_\lambda(\cross(n,3n))<1/200$, then
$$
\lim_{n\to\infty}\Pbf_\lambda(\cross(n,3n))=0.
$$
\end{enumerate}
\end{prop}

\begin{proof}
Set $N=\min\{n\ge1:\rho_\lambda(m)\le1/400\text{ for all $m\ge9n$ and }\lambda\in I\}$. Note that we may obtain a horizontal crossing on an $9n\times n$-rectangle from horizontal crossings of 4 overlapping $3n\times n$-rectangles and 3 vertical crossings of $n\times n$-rectangles. Consequently, using the union bound, we obtain that
$$
\Pbf_\lambda(\neg\cross(9n,n))\,\le\,7\Pbf_\lambda(\neg\cross(3n,n)).
$$
In addition, if the event $\cross(9n,3n)$ fails, then both $A_1=\cross(9n,n)$ and the translate $A_2$ of $A_1$ along the vector $(0,2n)$ has to fail too. Hence,
$$
\Pbf_\lambda(\neg\cross(9n,3n))\,\le\,\Pbf_\lambda(A_1^c\cap A_2^c)\,\le\,\Pbf_\lambda(A_1^c)^2+\rho_\lambda(9n).
$$
Moreover, if $q_\lambda(n):=1-\Pbf_\lambda(\cross(3n,n))$, then
\begin{equation}\label{eq:contraction}
q_\lambda(3n)\le49q_\lambda(n)^2+\rho_\lambda(9n).
\end{equation}
Now, if for some $\lambda\in I$ and $n\ge N$ we have $q_\lambda(n)<1/200$, then iterated use of~\eqref{eq:contraction} gives that $q_\lambda(3^kn)<1/200$ for all $k\ge1$, and for any $\ell=1,2,\ldots,k$ we find that
$$
q_\lambda(3^{k}n)\,\le\,\frac{1}{2^\ell}q_\lambda(3^{k-\ell}n)+\sum_{j=1}^\ell\frac{1}{2^{j-1}}\rho_\lambda(3^{k+1-j}n)\,\le\,\frac{1}{2^\ell}+2\rho_\lambda(3^{k+1-\ell}n).
$$
Taking limits, first sending $k$ and then $\ell$ to infinity, shows that $q_\lambda(3^kn)\to0$ as $k\to\infty$.

Since for any $n'\in[n,3n]$ we have
$$
q_\lambda(n')\,\le\,\Pbf_\lambda(\neg\cross(9n,n))\,\le\,7q_\lambda(n),
$$
part~\emph{a)} follows.
Part~\emph{b)} is proved verbatim, by replacing $\cross(n,3n)$ with the dual event that $[0,n]\times[0,3n]\setminus\occ$ contains a continuous curve connecting the the top and bottom of the rectangle, i.e., a vertical vacant crossing of the rectangle $[0,n]\times[0,3n]$.
\end{proof}

Under a slightly stronger assumption the conclusion of Proposition~\ref{prop:finite-size} can be strengthened to the existence of an unbounded occupied component.

\begin{prop}\label{prop:infinite-size}
Assume that the law of the random field is invariant with respect to rotations by right angles, and that $\pi_\lambda(n)\to0$ and $\overline\pi(n)\le(\alpha\log n)^{-(1+\alpha)}$ for some $\alpha>0$. Then,
$$
\lim_{n\to\infty}\Pbf_\lambda(\cross(3n,n))=1\quad\Rightarrow\quad\Pbf_\lambda\big(0\stackrel{\occ}{\leftrightarrow}\infty\big)>0.
$$
\end{prop}

\begin{proof}
Let $q_\lambda(n)$ be defined as in the proof of Proposition~\ref{prop:finite-size}. Let $K=[-n/4,13n/4]\times[-n/4,5n/4]$, and denote by $\overline q_\lambda(n)$ the $\Pbf_\lambda$-probability that $\occ_K$ does not contain a horizontal crossing of $[0,3n]\times[0,n]$. By assumption we have $q_\lambda(n)\to0$ as $n\to\infty$, so since also $\pi_\lambda(n)\to0$ as $n\to\infty$ we conclude that
$$
\overline q_\lambda(n)\le q_\lambda(n)+3\pi_\lambda(n)\to0\quad\text{as }n\to\infty.
$$
Let $K'=[-n/4,37n/4]\times[-n/4,5n/4]$ and denote by $A_1$ the event that $[0,9n]\times[0,n]$ is crossed horizontally by $\occ_{K'}$. Let $A_2$ denote the translate of $A_1$ along the vector $(0,2n)$. We note that the two events $A_1$ and $A_2$ are conditionally independent given $\varphi$. Hence, by repeating the steps of the proof of Proposition~\ref{prop:finite-size} we obtain the following analogue of~\eqref{eq:contraction}
\begin{equation}\label{eq:bar-contraction}
\overline q_\lambda(3n)\le49\overline q_\lambda(n)^2+\overline\pi(9n).
\end{equation}
Let $k_0$ be such that $\overline q_\lambda(3^k)<1/200$ for all $k\ge k_0$. By iterated use of~\eqref{eq:bar-contraction} we obtain
$$
\overline q_\lambda(3^{k+k_0})\le\frac{1}{2^\ell}+2\overline\pi(3^{k+1+k_0-\ell})
$$
for $\ell=1,2,\ldots,k$. With $\ell=\lfloor k/2\rfloor$ we obtain, summing over $k\ge1$, that
\begin{equation}\label{eq:finite-sum}
\sum_{k\ge1}\overline q_\lambda(3^{k+k_0})\,\le\, 4+2\sum_{k\ge1}\overline\pi(3^{k/2+k_0})\,<\,\infty.
\end{equation}

Now we tile the first quadrant by rectangles of dimensions $3^{k+1}\times3^k$, alternating between horizontally and vertically, each with its lower left corner at the origin. By~\eqref{eq:finite-sum} and Borel-Cantelli, all but finitely many of these rectangles will be crossed in the hard directions almost surely. Since the crossings of two rectangles at consecutive scales have to intersect, the crossings together form an unbounded occupied component.
\end{proof}

\begin{proof}[{\bf Proof of Theorem~\ref{thm:nontrivial}}]
We first prove part~\emph{a)}. Since $\pi_\lambda(n)$ is increasing in $\lambda$ it follows from Proposition~\ref{prop:cor-bound} that $\rho_\lambda(n)\to0$ as $n\to\infty$ uniformly on compact sets. Let $I=[0,1]$ and let $N=N(I)$ be as in Proposition~\ref{prop:finite-size}. For any $n\ge N$ we have $\Pbf_0(\cross(n,3n))=0$. However, for $m\ge3n$ large, so that $\pi_1(m)<1/400$, and $\delta>0$ small, so that the $\Pr_\delta$-probability of finding a point within an $m\times m$-box is at most $1/400$, we can assure that
$$
\Pbf_\delta(\cross(n,3n))<1/200.
$$
Via Proposition~\ref{prop:finite-size}, we conclude that $\Pbf_\delta(\cross(n,3n))\to0$ as $n\to\infty$.
We then note that on the event that the origin is contained in an unbounded component, for each $n\ge1$ there exists an occupied crossing connecting the boundary of $B_\infty(0,n/2)$ with the boundary of $B_\infty(0,3n/2)$. However, this may only happen if at least one out of four copies of an $n\times3n$-rectangle is crossed the easy way. That is, using the union bound,
$$
\Pbf_\delta\big(0\stackrel{\occ}{\leftrightarrow}\infty\big)\,\le\,4\Pbf_\delta\big(\cross(n,3n)\big)\to0
$$
as $n\to\infty$. This ends the proof of part~\emph{a)}.

We now turn to part~\emph{b)}. Let $M$ be as in the statement of the theorem and let $\varphi_M(x):=\min\{\varphi(x),M\}$. Notice that $\occ(\varphi_M)\subseteq\occ(\varphi)$, so we may for the rest of the proof work with $\varphi_M$. For the field $\varphi_M$ we have $\pi_\lambda(n)=0$ for all large $n$, and hence that $\rho_\lambda(n)\to0$ uniformly in $\lambda\in[0,\infty)$. Let $N=N([0,\infty))$ be as in Proposition~\ref{prop:finite-size}, and choose $n\ge N$ and $\lambda$ large so that $\Pbf_\lambda(\cross_M(3n,n))>1-1/200$. By Proposition~\ref{prop:finite-size} we have for this $\lambda$ that $\Pbf_\lambda(\cross_M(3n,n))\to1$ as $n\to\infty$. Then, by Proposition~\ref{prop:infinite-size}, we have
$$
\Pbf_\lambda\big(0\stackrel{\occ(\varphi)}{\longleftrightarrow}\infty\big)\,\ge\,\Pbf_\lambda\big(0\stackrel{\occ(\varphi_M)}{\longleftrightarrow}\infty\big)\,>\,0,
$$
and thus that $\lambda_c<\infty$, as required.
\end{proof}

We end this section by asking for a sharp threshold.

\begin{quest}
Under what conditions does there exist $\lambda_0\in(0,\infty)$ so that for all $\kappa\in(0,\infty)$
\bea
\lim_{n\to\infty}\Pbf_\lambda(\cross(\kappa n,n))\,=\,\left\{
\begin{aligned}
&1 && \text{for }\lambda>\lambda_0,\\
&0 && \text{for }\lambda<\lambda_0,\\
\end{aligned}
\right.
\eea
and such that $\Pbf_{\lambda_0}(\cross(\kappa n,n))\in(c,1-c)$ for all $n\ge1$ and some $c=c(\kappa)>0$?
\end{quest}

Techniques developed in~\cite{ahltastei18a} may be used to give a partial answer to this question. For a large class of models these techniques give the existence of $\lambda_0\le\lambda_1$ such that $\Pbf_\lambda(\cross(\kappa n,n))\in(c,1-c)$ for $\lambda\in[\lambda_0,\lambda_1]$, whereas $\Pbf_\lambda(\cross(\kappa n,n))$ tends to either 0 or 1 outside the `critical' interval $[\lambda_0,\lambda_1]$. To show that $\lambda_0=\lambda_1$ seems in general to be harder.

We remark that one could additionally define $\lambda_c^\star$ as the point of transition between the existence and nonexistence of an unbounded vacant component.
Our proof of Theorem~\ref{thm:nontrivial} can be adapted to show that $\lambda_c^\star<\infty$ whenever $\overline\pi(n)$ tends to zero and condition~\eqref{eq:thm2} holds,
and that $\lambda_c^\star>0$ whenever $\pi_\lambda(n)$ and $\overline\pi(n)$ are bounded by $(\alpha\log n)^{-(1+\alpha)}$ for some $\lambda>0$, $\alpha>0$ and all $n\ge1$.
Assuming that $\overline\pi(n)\le(\alpha\log n)^{-(1+\alpha)}$, which additional conditions are generally needed for the inequalities $0<\lambda_c^\star\le\lambda_c<\infty$ to hold?

Finally we note that condition~\eqref{eq:thm2} of Theorem~\ref{thm:nontrivial} is not far from sharp. Indeed, under the additional assumptions that the field is bounded, invariant with respect to reflection in coordinate axes, and positively associated, then the condition in~\eqref{eq:thm2} is necessary. To see this, assume that~\eqref{eq:thm2} fails.
By Proposition~\ref{prop:finite-size} we then have, for each $\lambda>0$, that $\Pbf_\lambda(\cross(3n,n))<1-1/200$ for all large $n$. By Theorem~4.1 of~\cite{ahltastei18a} it follows that also $\sup_{n\ge1}\Pbf_\lambda(\cross(n,3n))<1$. A standard argument (see Corollary~4.4 of~\cite{ahltastei18a}) then shows that an infinite occupied cluster cannot exist at intensity $\lambda$, almost surely. Since $\lambda>0$ was arbitrary, we must have $\lambda_c=\infty$.

\section{Comparison between the critical parameters}\label{sec:box-model}

We will in this section provide an example of a one-parameter family of random fields for which, depending on the value of the parameter, the critical value $\lambda_c$ is either strictly less than, equal to or strictly greater than the corresponding threshold $\lambda_\Phi$. The fields we consider will be constructed from a partitioning of the plane into bounded cells, on each of which the field will be given value $a$ or $b$, according to independent coin flips. The bias of the coin will provide a parameter for the model.

Given $a,b\in(0,\infty)$ where $a<b$, let $\lambda_a$ and $\lambda_b$ denote the critical parameters corresponding to the constant fields $\varphi\equiv a$ and $\varphi\equiv b$, respectively. It is well known that both $\lambda_a$ and $\lambda_b$ are nondegenerate, and the strict inequality $\lambda_a>\lambda_b$ follows from a simple scaling argument. It is further known (see~\cite[Theorem~5.1]{roytan02}) that for every $p\in(0,1)$ the critical parameter for the i.i.d.\ model associated with $\Phi=p\delta_b+(1-p)\delta_a$ satisfies
$$
0<\lambda_b<\lambda_\Phi(p)<\lambda_a<\infty.
$$

We define a (family of) fields $\varphi:\R^2\to\{a,b\}$ as follows. Given $\mu\in(0,\infty)$ and $p\in[0,1]$, let $\xi$ be a homogeneous Poisson point process on $\R^2\times[0,1]$ of intensity $\mu$. Let $\{C(y):(y,z)\in\xi\}$ denote the Voronoi tessellation of ${\mathbb R}^2$ based on $\xi$, i.e., where
$$
C(y):=\big\{x\in\R^2:|x-y|\le|x-y'|\text{ for all }(y',z')\in\xi\big\}.
$$
Finally, for $x\in\Z^2$ we set $\varphi(x)=b$ if $x\in C(y)$ for some $(y,z)\in\xi$ with $z\le p$, and $\varphi(x)=a$ otherwise. Let $\lambda_c(\mu,p)$ denote the corresponding critical parameter. Clearly
$$
\lambda_b\,\le\,\lambda_c(\mu,p)\,\le\,\lambda_a.
$$
Moreover, since $\varphi$ is bounded we have $\pi_\lambda(n)=0$ for all $\lambda$ and large $n$. On the other hand, as the expected area of the Voronoi cells are inverse proportional to $\mu$ we find that $\overline\pi(n)=\overline\pi(\mu,n)$ scales as
\begin{equation}\label{eq:pibarscaling}
\overline\pi(\mu,n)=\overline\pi(1,\sqrt{\mu}n),
\end{equation}
for $p$ fixed. That $\overline\pi(1,n)$ decays super-exponentially fast in $n$, uniformly in $p$, is a consequence of well-known properties of Voronoi tessellations, see~\cite[Lemma~8.18]{bolrio06a}.

We shall examine the behaviour of the model both for $\mu$ small and large. When $\mu$ is large most cells contain at most one point, which suggests that the critical intensity should be close to the critical intensity for i.i.d.\ radii.

\begin{prop}\label{prop:box-small}
The critical density $\lambda_c:(0,\infty)\times[0,1]\to[\lambda_b,\lambda_a]$ satisfies
$$
\lim_{\mu\to\infty}\sup_{p\in[0,1]}\big|\lambda_c(\mu,p)-\lambda_\Phi(p)\big|=0.
$$
\end{prop}

\begin{proof}
Since $\pi_\lambda(n)=0$ for large $n$ and $\overline\pi(\mu,n)=\overline\pi(1,\sqrt{\mu}n)$, we obtain constants $\gamma>0$ and $N\ge1$ as in Proposition~\ref{prop:finite-size}, uniformly in $\lambda$, $p$ and $\mu\ge1$.
The first step will be to show that for every $\lambda^\ast>0$ and $n\ge1$ we have for sufficiently large $\mu$ that
\be\label{eq:iid-comparison}
\sup_{\lambda\in[0,\lambda^\ast]}\sup_{p\in[0,1]}\big|\Pbf_\lambda(\cross(3n,n))-\Pr_\lambda(\cross(3n,n))\big|\le\gamma/3.
\ee
The same conclusion will hold for $\cross(3n,n)$ exchanged for $\cross(n,3n)$.

In order to establish~\eqref{eq:iid-comparison} we will construct the two processes on the same probability space. Let $\eta$ and $\xi$ be homogeneous Poisson point processes on $\R^2\times\R_+\times[0,1]$ and $\R^2\times\R_+$, respectively, both with intensity one. A realization of the occupied region for the i.i.d.\ process at intensity $\lambda$ is obtained as
$$
\occ_1:=\bigcup_{(x,y,z)\in\eta:\,y\le\lambda}B\big(x,a+(b-a)\ind_{\{z\le p\}}\big).
$$
Consider the Voronoi tiling based on $\{u:(u,v)\in\xi,v\le\mu\}$ and let $z_0(u)$ denote the value of the $z$-coordinate of the point $(x,y,z)\in\eta\cap C(u)$ with least $y$-coordinate. We construct the field $\varphi:\R^2\to\{a,b\}$ where $\varphi(w)=b$ if $w\in C(u)$ for some $(u,v)\in\xi$ with $v\le\mu$ and $z_0(u)\le p$.
Now, set
$$
\occ_2:=\bigcup_{(x,y,z)\in\eta:\,y\le\lambda}B(x,\varphi(x)).
$$
Since the projections of $\eta$ on each of its coordinates are independent it is easy to see that $\occ_1$ and $\occ_2$ have the correct marginal distributions. It is also clear from the construction that $\occ_1$ and $\occ_2$ will coincide on a compact set $K$ if each point of $\eta$ with $x$-coordinate within distance $b$ from $K$ are contained in disjoint Voronoi cells. Given $\lambda^\ast>0$ and $n\ge1$ the probability for this to fail is uniformly small in $\lambda\in[0,\lambda^\ast]$ and $p\in[0,1]$ for large enough $\mu$, which proves~\eqref{eq:iid-comparison}.

Pick $\eps>0$. For each $p\in[0,1]$ take $n\ge N$ so that $\Pr_{\lambda_\Phi(p)+\eps}(\cross(3n,n))>1-\gamma/3$, which by~\eqref{eq:iid-comparison} implies that $\Pbf_{\lambda_\Phi(p)+\eps}(\cross(3n,n))>1-2\gamma/3$ for large $\mu$.
On the other hand we find for each $p\in[0,1]$ an $n\ge N$ so that $\Pr_{\lambda_\Phi(p)-\eps}(\cross(n,3n))<\gamma/3$, and hence that $\Pbf_{\lambda_\Phi(p)-\eps}(\cross(n,3n))<2\gamma/3$ for large enough $\mu$.
By Proposition~\ref{prop:finite-size} we conclude that as $\mu\to\infty$
$$
\lambda_c(\mu,p)\to\lambda_\Phi(p)
$$
point-wise. In order to obtain uniform convergence we first recall (see~\cite[Theorem~3.7]{meeroy96} or~\cite[Theorem~6.1]{ahltastei18a}) that $\lambda_\Phi(p)$ is continuous as a function of $p$, and therefore that also $\Pr_{\lambda_\Phi(p)\pm\eps}(\cross(3n,n))$ is continuous in $p$, as a slight shift in $\lambda$ is unlikely to (\emph{i})~add or remove any point; and (\emph{ii})~change the value to either point already present. In particular, for each $p\in[0,1]$ we may obtain an $n\ge N$ such that $\Pr_{\lambda_\Phi(p)+\eps}(\cross(3n,n))>1-2\gamma/3$ in a neighbourhood of $p$. Due to compactness of the interval $[0,1]$, taking a finite subcover, we may find a single $n\ge N$ for which $\Pr_{\lambda_\Phi(p)+\eps}(\cross(3n,n))>1-2\gamma/3$ holds uniformly in $p\in[0,1]$. Consequently, by~\eqref{eq:iid-comparison}, for this $n$ we have
$$
\inf_{p\in[0,1]}\Pbf_{\lambda_\Phi(p)+\eps}\big(\cross(3n,n)\big)>1-\gamma
$$
for all large $\mu$. Similarly we obtain $\sup_{p\in[0,1]}\Pbf_{\lambda_\Phi(p)-\eps}\big(\cross(n,3n)\big)<\gamma$ for large enough $\mu$. In conclusion,
$$
\sup_{p\in[0,1]}\big|\lambda_c(\mu,p)-\lambda_\Phi(p)\big|\le\eps
$$
for large enough $\mu$, as required.
\end{proof}

We next examine the behaviour when $\mu$ is small, and the corresponding Voronoi cells large.
Part~\emph{c)} of Theorem~\ref{thm:comparison} is an immediate consequence of the following result.

\begin{prop}\label{prop:box-large}
There exists $\beta\in[\frac12,1)$ such that $\lambda_c:(0,\infty)\times[0,1]\to[\lambda_b,\lambda_a]$ satisfies
$$
\lim_{\mu\to0}\lambda_c(\mu,p)=\left\{
\begin{aligned}
&\lambda_b && \text{for all }p>1-\beta,\\
&\lambda_a && \text{for all }p<\beta.
\end{aligned}
\right.
$$
In particular, for each sufficiently small $\mu>0$ we have
$$
\lambda_c(\mu,\beta/2)>\lambda_\Phi(\beta/2)\quad\text{and}\quad\lambda_c(\mu,1-\beta/2)<\lambda_\Phi(1-\beta/2).
$$
\end{prop}

\begin{proof}
Since $b>a$, the critical density $\lambda_c(\mu,p)$ is non-increasing in $p$. It will suffice to show that for some $p<1$ and every $\lambda>\lambda_b$ there exists $n\ge N$ such that for sufficiently small $\mu$
$$
\Pbf_\lambda\big(\cross(3n/\sqrt{\mu},n/\sqrt{\mu})\big)>1-\gamma.
$$
In this case Proposition~\ref{prop:finite-size} implies that $\lambda_c(\mu,p)\le\lambda$. Similarly, if for some $p>0$ and every $\lambda<\lambda_a$ there exists $n\ge N$ such that for all small $\mu$ we have
$$
\Pbf_\lambda\big(\cross(n/\sqrt{\mu},3n/\sqrt{\mu})\big)<\gamma,
$$
then $\lambda_c(\mu,p)\ge\lambda$. Since the two cases are similar we only prove the former.

To this end we introduce an auxiliary parameter $m\ge1$, and fix $\mu=1/m^2$. Tile the plane with $m\times m$-boxes centred at the points of $(m\Z)^2$, and notice that as $m$ varies the Voronoi tiling scales accordingly. Let $\omega\in\{0,1\}^{\Z^2}$ be defined by
\begin{equation*}
\omega_z:=\left\{
\begin{aligned}
& 1 && \text{if $\varphi(x)=b$ for every }x\in mz+[-m/2,m/2]^2,\\
& 0 && \text{otherwise}.
\end{aligned}
\right.
\end{equation*}
Consider a rectangle made up of $3n\times n$ boxes of dimension $m\times m$. Let $E_n$ denote the event that there is a horizontal crossing of this rectangle of boxes for which $\omega$ takes the value 1. Note how the probability of $E_n$ is independent of $m$.

Given $y\in\Z^2$ let $F_1(y)$ denote the event that there is an occupied horizontal crossing of the rectangle $my+[-m/4,5m/4]\times[-m/4,m/4]$, and occupied vertical crossings of $my+[-m/4,m/4]^2$ and $m(y+{\bf e}_1)+[-m/4,m/4]^2$. We define $F_2(y)$ similarly, that there is an occupied vertical crossing of $my+[-m/4,m/4]\times[-m/4,5m/4]$ and occupied horizontal crossings of $my+[-m/4,m/4]^2$ and $m(y+{\bf e}_2)+[-m/4,m/4]^2$.

Fix $\lambda>\lambda_b$ and $n\ge N$. For all large $m$ we will have for $i=1,2$ that
\begin{equation}\label{e.daeq1}
\Pbf_\lambda\big(F_i(y)\,\big|\,\omega_y=\omega_{y+{\bf e}_i}=1\big)>1-\gamma/(12n^2).
\end{equation}
As there are no more than $6n^2$ neighbouring pairs of $m\times m$-squares in an $3nm\times nm$-rectangle, equation~\eqref{e.daeq1} implies that for all large $m$
$$
\Pbf_\lambda\big(\cross(3nm,nm)\big|E_n\big)>1-\gamma/2.
$$
As $\overline\Pr(E_n)>1-\gamma/2$ for $p$ close enough to 1, we have
$$
\Pbf_\lambda\big(\cross(3nm,nm)\big)\,\ge\,\Pbf_\lambda\big(\cross(3nm,nm)\big|E_n\big)\overline\Pr(E_n)\,>\,(1-\gamma/2)^2\,>\,1-\gamma.
$$
So, for the chosen $\lambda>\lambda_b$ and $p<1$ we have, via Proposition~\ref{prop:finite-size}, that $\lambda_c(1/m^2,p)\le\lambda$ for all large $m$, as required.
\end{proof}

It seems reasonable to believe that Proposition~\ref{prop:box-large} should hold with $\beta=1/2$, and that $\lambda_c(\mu,p)$ exhibits a threshold behaviour around $p=1/2$. More intriguing, perhaps, is to understand how $\lambda_c(\mu,p)$ behaves at $p=1/2$.

\begin{quest}
Does $\lambda_c(\mu,1/2)$ converge as $\mu\to0$, and to which value?
\end{quest}

\newcommand{\noopsort}[1]{}\def\cprime{$'$}

{\small
\noindent
{\sc Daniel Ahlberg\\
Instituto Nacional de Matem\'atica Pura e Aplicada\\
Estrada Dona Castorina 110, 22460-320 Rio de Janeiro, Brasil\\
and\\
Department of Mathematics, Uppsala University\\
SE-751 06 Uppsala, Sweden\\
}

\noindent
{\sc Johan Tykesson\\
Department of Mathematics, Chalmers University of Technology\\
SE-412 96 Gothenburg, Sweden}\\
\noindent
{\sc and\\
Department of Mathematics, University of Gothenburg\\
SE-412 96 Gothenburg, Sweden}\\
}

\end{document}